%% file: relgrad.tex
\documentclass[12pt]{amsart}
\usepackage{amscd}
\usepackage{amssymb}
\usepackage[noadjust]{cite}
\usepackage{booktabs}
\usepackage{url}
\usepackage{hyphenat}
\usepackage{epsf}
\usepackage{mathtools}
\usepackage{ifpdf}
\usepackage{mathrsfs}
\usepackage{array}
\usepackage[T1]{fontenc}
\usepackage[utf8]{inputenc}
\ifpdf
  \usepackage[pdftex,lmargin=1in,rmargin=1in,tmargin=1in,bmargin=1in]{geometry}
  \usepackage[bookmarks=true, bookmarksopen=true,%
    bookmarksdepth=3,bookmarksopenlevel=2,%
    colorlinks=true,%
    linkcolor=blue,%
    citecolor=blue,%
    filecolor=blue,%
    menucolor=blue,%
    urlcolor=blue]{hyperref}
  \hypersetup{pdftitle={Relative Q-gradings from bordered Floer theory}}
  \hypersetup{pdfauthor={Robert Lipshitz, Peter S. Ozsváth, and Dylan
      P. Thurston}}
\else
  \usepackage[dvips,final]{graphicx}
  \usepackage[dvips,lmargin=1in,rmargin=1in,tmargin=1in,bmargin=1in]{geometry}
  \usepackage[draft]{hyperref}
\fi
\usepackage{tikz}
\usetikzlibrary{matrix,arrows}
\tikzset{cdlabel/.style={above,sloped,
    execute at begin node=$\scriptstyle,execute at end node=$}}
\tikzset{algarrow/.style={->, thick}}
\tikzset{blgarrow/.style={->, thick}}
\tikzset{clgarrow/.style={->, thick}}
\tikzset{tensoralgarrow/.style={double, double equal sign distance, -implies}}
\tikzset{tensorblgarrow/.style={double, double equal sign distance, -implies}}
\tikzset{tensorclgarrow/.style={double, double equal sign distance, -implies}}
\tikzset{modarrow/.style={->, dashed}}
\tikzset{Amodar/.style={->, dashed}}
\tikzset{Dmodar/.style={->, dashed}}

\input{macros}
\input{defs}

\begin{document}
\title[Relative $\QQ$-gradings]{Relative $\QQ$-gradings from bordered Floer theory}

\author[Lipshitz]{Robert Lipshitz}
\thanks{RL was supported by an NSF Grant DMS-0905796 and a Sloan Research
  Fellowship.}
\address{Department of Mathematics, Columbia University\\
  New York, NY 10027}
\email{lipshitz@math.columbia.edu}

\author[Ozsv\'ath]{Peter Ozsv\'ath}
\thanks{PSO was supported by NSF grant number DMS-0804121.}
\address {Department of Mathematics, Princeton University\\ New
  Jersey, 08544}
\email {petero@math.princeton.edu}

\author[Thurston]{Dylan~P.~Thurston}
\thanks {DPT was supported by NSF grant number DMS-1008049 and a Sloan Research Fellowship.}
\address{Department of Mathematics,
         UC Berkeley,
         970 Evans Hall,
         Berkeley, CA 94720}
\email{dpt@math.berkeley.edu}

\begin{abstract}
  In this paper we show how to recover the relative $\QQ$-grading in
  Heegaard Floer homology from the noncommutative grading on bordered
  Floer homology.
\end{abstract}

\maketitle

\tableofcontents

\section{Introduction}

Heegaard Floer homology, introduced by the second author and
Z. Szab{\'o}, is an invariant for a three-manifold equipped with a
$\SpinC$ structure~\cite{OS04:HolomorphicDisks}.  
Heegaard Floer homology is
defined as Lagrangian intersection Floer homology groups of certain
Lagrangians in a symmetric product of a Riemann surface, and as such is
most naturally are only relatively, cyclicly graded. Indeed, the
Heegaard Floer homology of a three-manifold equipped with the $\SpinC$
structure $\spinc$ is graded by the group $\ZZ/\divis(c_1(\spinc))$,
where $\divis(c_1(\spinc))$ denotes the divisibility of the first
Chern class of the $\SpinC$ structure~$\spinc$. In particular, if the
first Chern class of $\spinc$ is torsion, then the corresponding
Heegaard Floer homology is relatively $\ZZ$-graded.

With the help of the functorality properties of Heegaard Floer
homology, the relative $\ZZ$-grading on Heegaard Floer homology can be
lifted to an absolute $\QQ$-grading when the underlying $\SpinC$
structure is torsion~\cite{OS06:HolDiskFour}. (Compare
Fr{\o}yshov~\cite{Froyshov04}.) This absolute $\QQ$-grading contains
subtle
topological information; 
for a beautiful recent application, see~\cite{Greene:MutationLatticesGraphs}.
%~\cite{AbsGraded,Owens08:unknotting,GreenJabuka11:slice,ManolescuOwens07:Concordance,
%  JabukaNaik07:Concordance,Stipsicz08:mubar,
%  Peters:Concordance,Greene:cabling,Greene:realization,GreeneWatson:Turaev,GilmerLivingston11:NonOrientable}. %
%\smargin{RL: more? Fewer?}%
Although, by work of Sarkar-Wang~\cite{SarkarWang07:ComputingHFhat} and
Sarkar~\cite{Sarkar06:IndexTriangles}, the
absolute $\QQ$-grading is algorithmically computable, no simple
formula is known, and it remains somewhat mysterious.

By contrast, the relative $\QQ$-grading induced by the absolute
$\QQ$-grading is much simpler. In this paper, we show how to use
bordered Floer homology to compute this relative $\QQ$-grading between
different torsion $\SpinC$ structures, by
decomposing a $3$-manifold along a connected surface; it turns out
that the non-commutative grading on bordered Floer homology contains
the necessary information. (Another way of computing the relative
$\QQ$-grading, using covering spaces, was given by D.~Lee and the
first author~\cite{LipshitzLee08:RelGrad}.)

Finally, note that Heegaard Floer homology has several variants,
$\HFa$, $\HFm$, $\HFinf$, and $\HFp$.  Although we focus on the
relative $\QQ$-grading on $\HFa$
(as that is the version with a corresponding bordered theory),
this determines the relative $\QQ$-grading on $\HF^+$,
$\HF^-$ and $\HF^\infty$, via the exact triangles
\[
\cdots\longrightarrow \HFa(Y)\longrightarrow \HF^+(Y)\stackrel{\cdot U}{\longrightarrow}
\HF^+(Y)\stackrel{[1]}{\longrightarrow} \cdots
\]
and
\[
\cdots\longrightarrow \HF^-(Y)\longrightarrow
\HF^\infty(Y)\longrightarrow \HF^+(Y)\stackrel{[1]}{\longrightarrow}
\cdots.
\]

\vspace{\baselineskip} 
\noindent
\textbf{Acknowledgements.} This paper was
started while we were visiting the Mathematical Sciences Research
Institute (MSRI) and finished while the first two authors were visiting the
Simons Center for Geometry and Physics and the third author was
visiting U.\ C.\ Berkeley and MSRI. We thank all three institutions for their hospitality.

\section{Background}
\subsection{The relative $\QQ$-grading}\label{sec:back:Q-grad}
The absolute $\QQ$-grading on $\HFa(Y,\spinc)$ is defined as
follows. Choose a $\SpinC$ nullcobordism $(W^4,\mathfrak{t})$ of
$(Y,\spinc)$. Associated to the cobordism $W$ is a map
$\hat{F}_{W,\spinc}\co \HFa(S^3,\spinc_0)\to \HFa(Y,\spinc)$. The
absolute grading on $\HFa(Y,\spinc)$ is characterized by the property
that the generator of $\HFa(S^3)\cong\ZZ$ lies in degree $0$ and the
map $F$ has degree
\[
\frac{c_1(\mathfrak{t})^2-2\chi(W)-3\sigma(W)}{4}.
\]
(Actually, since $\hat{F}_{W,\spinc}$ might be trivial on homology, it
is more accurate to say the grading is characterized by the property
that Maslov index $0$ triangles in the definition of
$\hat{F}_{W,\spinc}$ have this degree.)
See~\cite[Section 7]{OS06:HolDiskFour}.

The paper~\cite{LipshitzLee08:RelGrad} shows that the relative
$\QQ$-grading can be reformulated as follows. Suppose that $\x$ and
$\y$ are generators for $\CFa(Y)$ (computed via some pointed Heegaard diagram
$\HD=(\Sigma,\alphas,\betas,z)$) so that $c_1(\spinc(\x))$ and $c_1(\spinc(\y))$ are
torsion. Then there is a finite-order covering space $p\co
\widetilde{Y}\to Y$ so that
$p^*\spinc(\x)=p^*\spinc(\y)$~\cite[Corollary
2.10]{LipshitzLee08:RelGrad}. The Heegaard diagram $\HD$
for $Y$ lifts to a (multi-pointed) Heegaard diagram $\widetilde{\HD}$
for $\widetilde{Y}$. The generators $\x$ and $\y$ have preimages
$p^{-1}(\x)$ and $p^{-1}(\y)$ in $\widetilde{\HD}$ so that
$\spinc(p^{-1}(\x))=p^*\spinc(\x)$ and
$\spinc(p^{-1}(\y))=p^*\spinc(\y)$. Thus, $p^{-1}(\x)$ and
$p^{-1}(\y)$ have a well-defined $\ZZ$-grading difference. Then
\[
\gr_\QQ(\y)-\gr_\QQ(\x)=\frac{1}{n}\gr_{\ZZ}(p^*(\x),p^*(\y)),
\]
where $n$ is the order of the cover $\widetilde{Y}\to Y$.

More concretely, even though $\x$ and $\y$ can not necessarily be
connected by a domain in $\pi_2(\x,\y)$, if we allow rational
multiples of the regions in $\Sigma$ then they can be connected. That
is, let $\pi_2^\QQ(\x,\y)$ denote the set of rational linear
combinations $B$ of components of $\Sigma\setminus(\alphas\cup\betas)$
connecting $\x$ to $\y$ (i.e., $\bdy(\bdy B\cap\alphas)=-\bdy(\bdy
B\cap\betas)=\y-\x$).  If $c_1(\spinc(\x))-c_1(\spinc(\y))$ is torsion
then $\pi_2^\QQ(\x,\y)$ is nonempty, and
\[
\gr_\QQ(\y)-\gr_\QQ(\x)=e(B)+n_\x(B)+n_\y(B)
\]
for any $B\in\pi_2^\QQ(\x,\y)$. Here, $e(B)$ denotes the Euler measure
of $B$ and $n_\x(B)$ and $n_\y(B)$ denote the point measure of $B$ at
the points $\x$ and $\y$; compare \cite[Section
4.2]{Lipshitz06:CylindricalHF}. These quantities on the right are
thought of as the natural $\QQ$-linear extensions of the usual
integer-valued analogues.  This agrees with the formulas in
\cite[Section 2.3]{LipshitzLee08:RelGrad}.

\subsection{The structure of bordered Floer theory}\label{sec:back:borderedFloer}
Bordered Floer theory assigns to a surface $F=F(\PMC)$ represented by
a \emph{pointed matched circle} $\PMC$ (Figure~\ref{fig:pmc-to-surf})
a \dg algebra $\Alg=\Alg(\PMC)$~\cite[Section 3]{LOT1}. To a
$3$-manifold $Y$ with boundary parameterized by $F(\PMC)$ it
associates invariants $\CFAa(Y)_{\Alg(\PMC)}$, a right $\Ainf$-module
over $\Alg(\PMC)$, and $\lsup{\Alg(-\PMC)}\CFDa(Y)$, a left,
projective \dg module over $\Alg(-\PMC)$~\cite[Sections 6,
7]{LOT1}. Each of $\CFAa(Y)$ and $\CFDa(Y)$ are well-defined up to
homotopy equivalence. These modules are related to the invariants of a closed
$3$-manifold by a \emph{pairing theorem}:
\begin{citethm}\label{thm:pairing}\cite[Theorem 1.3]{LOT1}
  If $Y_1$ and $Y_2$ are $3$-manifolds with boundaries parameterized
  by $F(\PMC)$ and $-F(\PMC)$ respectively then
  \[
  \CFa(Y_1\cup_FY_2)\simeq \CFAa(Y_1)\DTP_{\Alg(\PMC)}\CFDa(Y_2).
  \]
\end{citethm}
Here, $\DTP_{\Alg(\PMC)}$ denotes the $\Ainf$-tensor product over
$\Alg(\PMC)$. There is a particularly convenient model $\DT$ for the
$\Ainf$-tensor product so that $\CFa(Y_1\cup_FY_2)$ is actually
isomorphic as an $\Field$-vector space to $\CFAa(Y_1)\DT\CFDa(Y_2)$
(for corresponding choices of auxiliary data, as discussed below).

\begin{figure}
  \centering
  \includegraphics{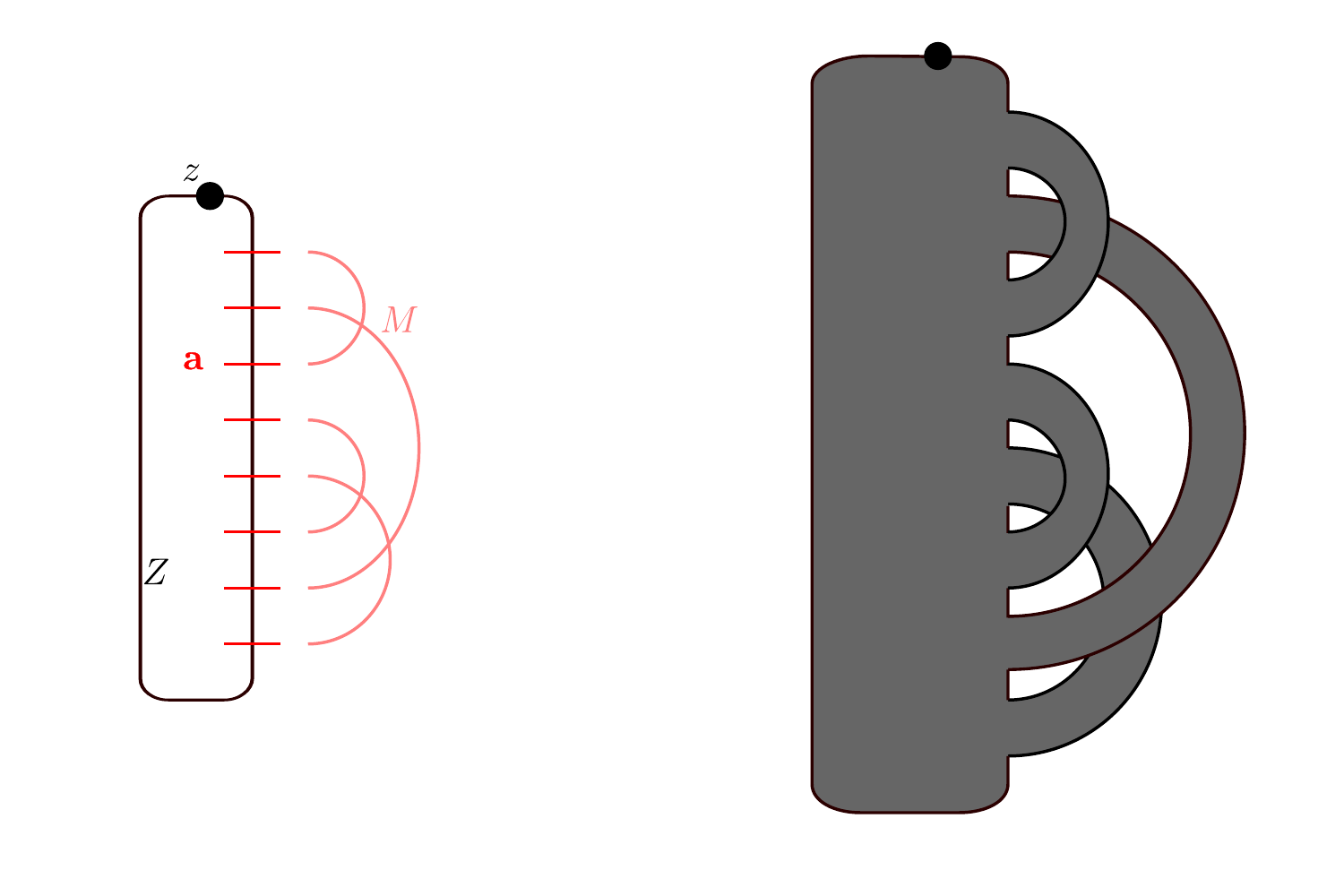}
  \caption{\textbf{The surface represented by a pointed matched
      circle.} Left: a pointed matched circle $\PMC$, consisting of a
    circle $Z$, $4k$ points $\mathbf{a}$ (the case $k=2$ is shown)
    matched in pairs via a matching $M$, and with a basepoint
    $z$. Right: the surface with boundary $F^\circ(\PMC)$ represented
    by $\PMC$. The surface $F(\PMC)$ represented by $\PMC$ is obtained
    by gluing a disk to $F^\circ(\PMC)$. As such, it contains a
    distinguished disk and a basepoint on the boundary of that disk.}
  \label{fig:pmc-to-surf}
\end{figure}

The isomorphism in Theorem~\ref{thm:pairing} is an isomorphism of
relatively graded groups, in an appropriate sense. This will be
discussed further in Section~\ref{sec:BorderedGr}.

For the purposes of this paper, we will use the following basic facts
about $\Alg(\PMC)$, $\CFAa(Y)$ and $\CFDa(Y)$:
\begin{itemize}
\item The invariants $\CFAa(Y)$ and $\CFDa(Y)$ are defined in terms of
  a bordered Heegaard diagram $\HD=(\Sigma,\alphas,\betas,z)$ for
  $Y$. Here, $\Sigma$ is a compact, orientable surface of some genus
  $g$ with one boundary component; $\alphas$ consists of
  pairwise-disjoint embedded arcs $\alphas^a$ and circles $\alphas^c$
  in $\Sigma$, with $\bdy\alphas^a\subset \bdy\Sigma$, while $\betas$
  consists of embedded circles only; and $z$ is a basepoint in
  $\bdy\Sigma$, not lying on any $\alpha$-arc. See
  Figure~\ref{fig:bordered-HD} for an example, and~\cite[Section
  4]{LOT1} for more details.
\item If the bordered Heegaard diagrams $\HD_1$ and $H_2$ represent
  $Y_1$ and $Y_2$, respectively, and $\bdy Y_1=F(\PMC)=-\bdy Y_2$,
  then $\HD=\HD_1\cup_\bdy \HD_2$ represents $Y=Y_1\cup_\bdy Y_2$.
\item Given a genus $g$ bordered Heegaard diagram $\HD$ representing
  $Y$, the modules $\CFDa(Y)$ and $\CFAa(Y)$ are generated by all sets
  $\x=\{x_1,\dots,x_g\}$ of $g$ points in $\alphas\cap\betas$ so that
  exactly one $x_i$ lies on each $\alpha$- or $\beta$-circle and at
  most one $x_i$ lies on each $\alpha$-arc. (Again, see
  Figure~\ref{fig:bordered-HD}.) Let $\Gen(\HD)$ denote the set of
  generators $\x$ in $\HD$.
\item Given generators $\x$ and $\y$ in $\Gen(\HD)$, a \emph{domain
    connecting $\x$ to $\y$} is a linear combination $B$ of components
  of $\Sigma\setminus(\alphas\cup\betas)$ so that $\bdy((\bdy
  B)\cap\betas)=\x-\y$ and $\bdy((\bdy
  B)\cap(\alphas\cup\bdy\Sigma))=\y-\x$. (See
  Figure~\ref{fig:bordered-HD}.) Let $\pi_2(\x,\y)$ denote the set of
  domains connecting $\x$ to $\y$. For $B\in\pi_2(\x,\y)$, let
  $\bdy^\alpha B=(\bdy B)\cap \alphas$, $\bdy^\beta B=(\bdy
  B)\cap\betas$ and $\bdy^\bdy B=(\bdy B)\cap (\bdy \Sigma)$.
\item Given a bordered Heegaard diagram $\HD$ for $Y$, associated to
  each generator $\x\in\Gen(\HD)$ is a $\SpinC$-structure
  $\spinc(\x)$ on~$Y$. The modules $\CFDa(Y)$ and $\CFAa(Y)$ decompose
  according to these $\SpinC$-structures,
  $\CFDa(Y)=\bigoplus_{\spinc\in\SpinC(Y)}\CFDa(Y;\spinc)$ and
  $\CFAa(Y)=\bigoplus_{\spinc\in\SpinC(Y)}\CFAa(Y;\spinc)$. Let
  $\Gen(\HD,\spinc)=\{\x\in\Gen(\HD)\mid \spinc(\x)=\spinc\}$ denote
  the set of generators for $\CFAa(Y,\spinc)$ and $\CFDa(Y,\spinc)$.
\item Given bordered Heegaard diagrams $\HD_1$ and $\HD_2$ with $\bdy
  \HD_1=-\bdy\HD_2$, let $\HD=\HD_1\cup_\bdy \HD_2$. 
  There is an obvious
  embedding $\Gen(\HD)\to \Gen(\HD_1)\times\Gen(\HD_2)$ of the set of
  generators $\Gen(\HD)$ of $\CFa(\HD)$. The image of
  this embedding is the set of pairs
  $(\x_1,\x_2)\in\Gen(\HD_1)\times\Gen(\HD_2)$ so that $\x_1$ and
  $\x_2$ occupy complementary $\alpha$-arcs. It turns out that these
  are exactly the generators of $\CFAa(\HD_1)\DT\CFDa(\HD_2)$.
\end{itemize}

\begin{figure}
  \centering
  \includegraphics{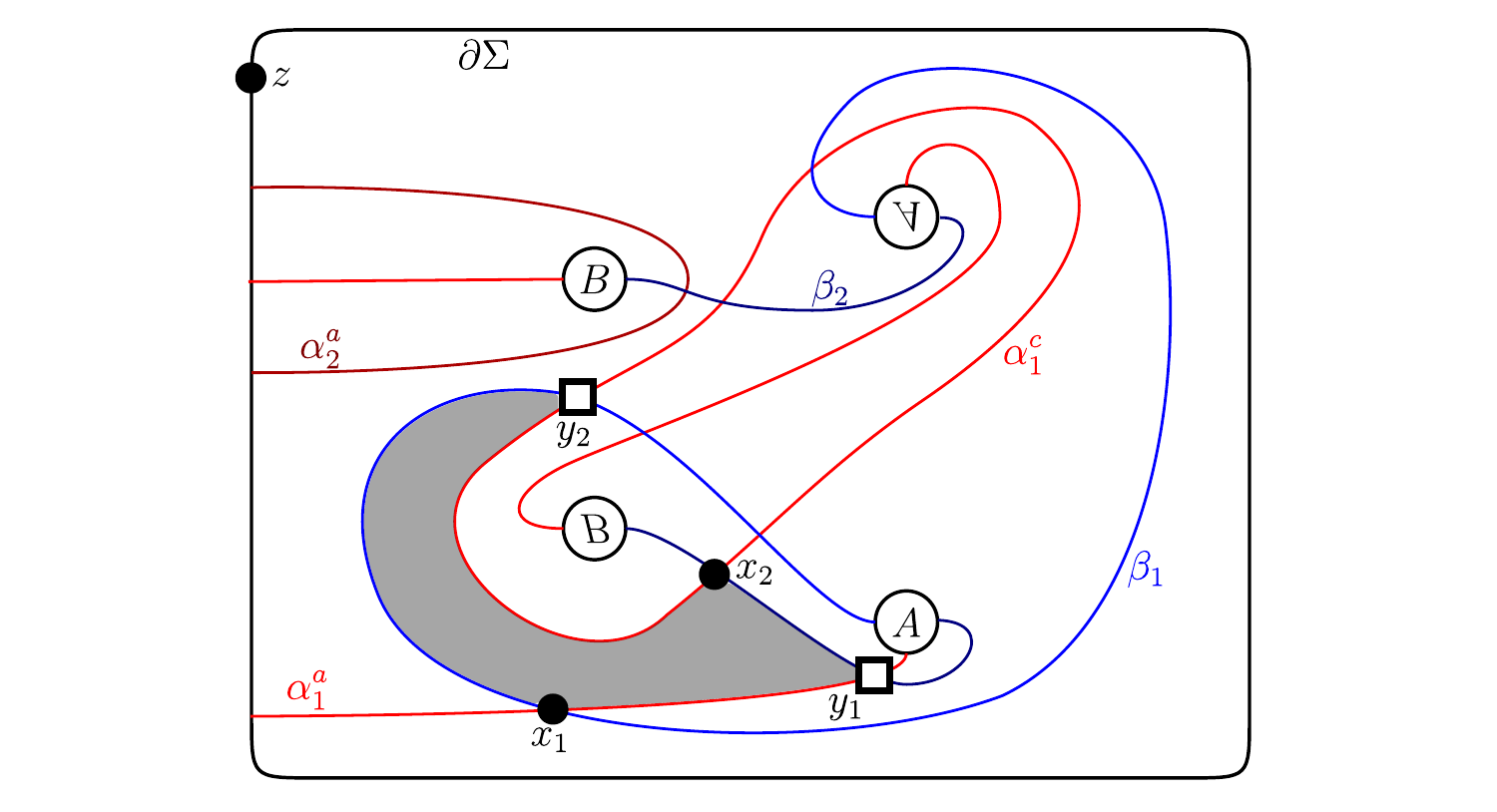}
  \caption{\textbf{A bordered Heegaard diagram.} The circles labeled
    $A$ and $B$ are connected by tubes. This bordered Heegaard diagram
    represents the trefoil complement (with a particular
    parameterization of its boundary). Two generators in $\Gen(\HD)$
    are marked, one by solid disks and the other by empty squares. A
    domain $B\in\pi_2(\x,\y)$ connecting $\x$ and $\y$ is also shown,
    shaded.}
  \label{fig:bordered-HD}
\end{figure}

There is an extension of bordered Floer theory to manifolds with two
boundary components, which are assigned various types of
bimodules~\cite{LOT2}. The generalizations of the results of this
paper to the bimodule case are straightforward, and we shall not
discuss them.

\subsection{The (non-commutative) grading in bordered Floer}\label{sec:BorderedGr}
As noted in the introduction, the grading on bordered Floer homology
is non-commutative.
\begin{definition}
  Let $G$ be a group and $\lambda$ a central element in $G$. If $\Alg$
  is a differential algebra, a \emph{grading of $\Alg$ by $G$}
  consists of a
  decomposition (as abelian groups) $\Alg=\bigoplus_{g\in G}\Alg_g$ of
  $\Alg$ into homogeneous parts so that for any homogeneous algebra
  elements $a$ and $b$,
  \begin{align}
      \gr(ab)&=\gr(a)\gr(b)&\text{if }&ab\neq 0\\
      \gr(\bdy a)&=\lambda^{-1}\gr(a)& \text{if }&\bdy a\neq 0.
  \end{align}

  Let $S$ be a left $G$-set. If $\lsub{\Alg}M$ is a left differential
  $\Alg$-module, a \emph{grading of $\lsub{\Alg}M$ by $S$} consists of a
  decomposition (as abelian groups) $M=\bigoplus_{s\in S}M_s$ of $M$
  into homogeneous parts such that for homogeneous elements $x\in M$
  and $a\in \Alg$,
  \begin{align}
    \label{eq:abstract-mod-gr}
    \gr(ax)&=\gr(a)\gr(x) & \text{if }&ax\neq 0\\
    \gr(\bdy x)&=\lambda^{-1}\gr(x) & \text{if }&\bdy x\neq 0.\label{eq:abstract-mod-gr-2}
  \end{align}
  More generally, if $S$ is an $\Ainf$-module over $\Alg$,
  Equations~(\ref{eq:abstract-mod-gr})
  and~(\ref{eq:abstract-mod-gr-2}) become
  \[
  \gr(m_{k+1}(a_1,\dots,a_k,x))=\lambda^{k-1}\gr(a_1)\cdots\gr(a_k)\gr(x),
  \]
  if $m_{k+1}(a_1,\dots,a_k,x)\neq 0$.

  Gradings on right modules by right $G$-sets are defined similarly.
\end{definition}

In the case of bordered Floer homology, for each pointed matched
circle $\PMC=(Z,\mathbf{a},M,z)$ there is a group $G'(\PMC)$ so that
$\Alg(\PMC)$ is graded by $G'(\PMC)$. The group $G'(\PMC)$ is a
$\ZZ$-central extension of $H_1(Z,\mathbf{a})$. To specify it, given a
point $p\in \mathbf{a}$ and homology class $c\in H_1(Z,\mathbf{a})$,
define $\mu(a,p)$ to be the average local multiplicity of $a$ near
$p$. Extend $\mu$ to a bilinear map $H_1(Z,\mathbf{a})\otimes
H_0(\mathbf{a})\to \frac{1}{2}\ZZ$. Then $G'$ is defined by the
commutation relation
\[
g\cdot h = \lambda^{2\mu([h],\bdy [g])}h\cdot g,
\]
where $[\cdot]\co G'(\PMC)\to H_1(Z,\mathbf{a})$ is the canonical
projection and $\lambda$ is a generator of the central $\ZZ$.

Explicitly, we can write elements of $G'(\PMC)$ as pairs $(m,c)$ where
$m\in\QQ$ and $c\in H_1(Z,\mathbf{a})$, with multiplication given by
\begin{equation}\label{eq:GpMult}
(m_1,\alpha_1)\cdot(m_2,\alpha_2)
=(m_1+m_2+m(\alpha_2,\bdy\alpha_1),\alpha_1+\alpha_2).
\end{equation}
The group $G'(\PMC)$ is generated by the elements $\lambda=(1;0)$ and
$(-\frac{1}{2},[i,i+1])$.

The grading on $\CFAa(\HD,\spinc)$ is given as follows.  Given a
Heegaard diagram $\HD$ with $\bdy\HD=\PMC$ and a $\SpinC$-structure
$\spinc$ on $Y=Y(\HD)$, fix a base generator
$\x_0\in\Gen(\HD,\spinc)$. For a domain $B\in\pi_2(\x,\y)$ define
\begin{equation}
  \label{eq:g-of-B}
  g'(B)=(-e(B)-n_\x(B)-n_\y(B),\bdy^\bdy B)\in\bigGroup(\PMC).
\end{equation}
If $B_1\in\pi_2(\x,\y)$ and $B_2\in\pi_2(\y,\z)$, let
$B_1*B_2\in\pi_2(\x,\z)$ denote the concatenation of $B_1$ and $B_2$.
Then $g'(B_1*B_2)=g'(B_1)\cdot g'(B_2)$~\cite[Lemma 6.15]{LOT1}.  In
particular, $P'(\x_0)=\{g'(B)\mid B\in\pi_2(\x_0,\x_0)\}$ is a subgroup
of $\bigGroup(\PMC)$. The module $\CFAa(\HD,\spinc)$ is graded by the
right $\bigGroup(\PMC)$-set 
% $\lsub{\HD}S=P'(\x_0)\backslash \bigGroup(\PMC)$.
$\ABigGrSet(\HD,\spinc)\coloneqq P'(\x_0) \backslash \bigGroup(4k)$.
(This construction depends on $\x_0$, but different choices of $\x_0$
give canonically isomorphic grading sets; see \cite[Section 10.3]{LOT1}.)
The grading of an element $\x\in\Gen(\HD,\spinc)$ is given by 
$g(B)$ for any $B\in\pi_2(\x_0,\x)$, thought of as an element of  the
coset space
$\ABigGrSet(\HD,\spinc)\coloneqq P'(\x_0) \backslash \bigGroup(4k)$.

The invariant $\CFDa(\HD)$ is a module over $\Alg(-\bdy\HD)$ rather
than $\bdy\HD$. So, in grading $\CFDa(\HD)$ we will use the anti-homomorphism $R\co
\bigGroup(-\PMC)\to \bigGroup(\PMC)$ given by
$R(j,\alpha)=(j,r_*(\alpha))$ where $r\co {-Z}\to Z$ is the
(orientation-reversing) identity map. 
The grading on $\CFDa(\HD,\spinc)$ is then defined similarly to the
grading on $\CFAa(\HD,\spinc)$, except that
the left module $\CFDa(\HD,\spinc)$ is graded by the left $\bigGroup$-set
$\DBigGrSet(\HD,\spinc)\coloneqq \bigGroup(4k)/R(P'(\x_0))$, and
%$S_\HD=\bigGroup(\PMC)/P'(\x_0)$. 
the grading of an element $x\in\Gen(\HD,\spinc)$ is given by 
(the equivalence class of) $R(g(B))$ for any $B\in\pi_2(\x_0,\x)$.

The tensor product $\CFAa(\HD_1,\spinc_1)\DT\CFDa(\HD_2,\spinc_2)$ is
graded by the amalgamated product of the grading sets
% $\lsub{\HD_1}S\times_{\bigGroup}S_{\HD_2}$; 
$\ABigGrSet(\HD_1)\times_{\bigGroup}\DBigGrSet(\HD_2)$;
the grading of
$\x_1\otimes\x_2$ is
$\gr'(\x_1\otimes\x_2)=(\gr'(\x_1),\gr'(\x_2))$.
(In fact, certain results are cleaner if one works instead with a certain subset of this amalgamated
product that contains the gradings of all tensor products of
generators; compare Theorem~\ref{thm:pairing-graded}, below,
and \cite[Theorem~10.43]{LOT1}.)
Note that since $\lambda$
is central in $\bigGroup$, the set
% $\lsub{\HD_1}S\times_{\bigGroup}S_{\HD_2}$ 
$\ABigGrSet(\HD_1)\times_{\bigGroup}\DBigGrSet(\HD_2)$
retains an action by $\lambda$, which we will think of as a $\ZZ$-action.

A graded version of the pairing theorem states:
\begin{citethm}\label{thm:pairing-graded}\cite[Theorem 9.33]{LOT1}
  If $Y_1$ and $Y_2$ are $3$-manifolds with boundaries parameterized by $F$ and $-F$ respectively then there is a map
  \[
  \Phi\co \CFAa(Y_1)\DT\CFDa(Y_2)\to \CFa(Y_1\cup_FY_2)
  \]
  such that:
  \begin{enumerate}
  \item $\Phi$ is a homotopy equivalence.
  \item Given generators $\x_1\otimes\x_2$ and $\y_1\otimes\y_2$ for
    $\CFAa(Y_1)\DT\CFDa(Y_2)$,
    $\spinc(\Phi(\x_1\otimes\x_2))=\spinc(\Phi(\y_1\otimes\y_2))$ if
    and only if:
    \begin{itemize}
    \item $\spinc(\x_1)=\spinc(\y_1)\eqqcolon \spinc_1$,
    \item $\spinc(\x_2)=\spinc(\y_2)\eqqcolon \spinc_2$, and
    \item the generators $g'(\x_1)\times_{\bigGroup}g'(\x_2)$ and
      $g'(\y_1)\times_{\bigGroup}g'(\y_2)$ lie in the same $\ZZ$-orbit
      of
      $\ABigGrSet(\HD_1,\spinc_1)\times_{\bigGroup}\DBigGrSet(\HD_2,\spinc_2)$.
      % $\lsub{\HD_1}S\times_{\bigGroup}S_{\HD_2}$.
    \end{itemize}
  \item If $\spinc(\x_1\otimes\x_2)=\spinc(\y_1\otimes\y_2)$ then
    \[    
    g'(\y_1)\times_{\bigGroup}g'(\y_2)=\lambda^{\gr(\Phi(\x_1\otimes\x_2),\Phi(\y_1\otimes\y_2))}g'(\x_1)\times_{\bigGroup}g'(\x_2).
    \]
  \end{enumerate}
\end{citethm}

For this paper, we will use a slightly larger grading group, and
corresponding grading sets. Given a pointed matched circle
$\PMC=(Z,\mathbf{a},M,z)$, let $\bigGroup_\QQ(\PMC)$ denote the
$\QQ$-central extension of $H_1(Z,\mathbf{a};\QQ)$ with multiplication
given by
\[
(m_1,\alpha_1)\cdot(m_2,\alpha_2)
=(m_1+m_2+m(\alpha_2,\bdy\alpha_1),\alpha_1+\alpha_2),
\]
i.e., the same formula as Equation~(\ref{eq:GpMult}).

There is an obvious inclusion $\bigGroup\to \bigGroup_\QQ$, so the
$\bigGroup$-grading on $\Alg(\PMC)$ induces a $\bigGroup_\QQ$-grading
on $\Alg(\PMC)$. Also, note that for $g\in\bigGroup_\QQ$ and $q\in\QQ$
there is a well-defined element $q\cdot g\in\bigGroup_\QQ$ obtained by
multiplying all of the coefficients in $g$ by $q$.

If $\HD$ is a Heegaard diagram with $\bdy\HD=\PMC$ (respectively
$\bdy\HD=-\PMC$), we can define a $\bigGroup_\QQ$-grading on
$\CFAa(\HD)$ (respectively $\CFDa(\HD)$) using
Formula~(\ref{eq:g-of-B}). Given $\x\in\Gen(\HD)$ let $P'_\QQ(\x)$ denote
the subgroup of $\bigGroup_\QQ$ generated by $\{q\cdot g'(B)\mid
B\in\pi_2(\x,\x),\ q\in\QQ\}$. Fix a base generator
$\x_0\in\Gen(\HD,\spinc)$.  For any $\x\in\Gen(\HD,\spinc)$ choose a
$B\in\pi_2(\x_0,\x)$ and define $\gr'_\QQ(\x)=g'(B)$ (respectively 
$\gr'_\QQ(\x)=R(g'(B))$), viewed as an
element of
$\AQBigGrSet(\HD,\spinc)\coloneqq P'_\QQ(\x_0) \backslash \bigGroup_\QQ(4k)$
% $\lsub{\QQ}S_\HD\coloneqq \bigGroup_\QQ/P'(\x_0)$
(respectively $\DQBigGrSet(\HD,\spinc)\coloneqq
\bigGroup_\QQ(4k)/R(P'_\QQ(\x_0))$).

% Again, the discussion for $\CFAa$ is analogous, using right $\bigGroup_\QQ$-sets rather than left $\bigGroup_\QQ$-sets.

There is also a refined grading on the algebra, by a group $G$ which
is a $\ZZ$-central extension of $H_1(F(\PMC))$, and corresponding
gradings on the modules; see~\cite[Section 3.3]{LOT1} or~\cite[Section
3.1.1]{LOT2}. Generally we will work with the larger grading group in
this paper, but see also Remark~\ref{rmk:Q-for-small-gr-grp}.

\section{From bordered Floer to the relative $\QQ$-grading}
\begin{theorem}\label{thm:compute-Q-gr}
  Suppose that $Y$ is a closed $3$-manifold, decomposed along a
  connected surface as $Y=Y_1\cup_F Y_2$. Let
  $\HD=\HD_1\cup_\PMC\HD_2$ be a corresponding decomposition of a
  Heegaard diagram for $Y$. Suppose that $\x,\y\in\Gen(\HD)$ are such
  that $\spinc(\x)$ and $\spinc(\y)$ are torsion, and
  $\spinc(\x)|_{Y_i}=\spinc(\y)|_{Y_i}\eqqcolon \spinc_i$ for $i=1,2$. Write
  $\x=\x_1\otimes\x_2$ and $\y=\y_1\otimes\y_2$, where $\x_i$ and
  $\y_i$ are in $\Gen(\HD_i)$. Then
  \begin{enumerate}
  \item the generators $\gr'_\QQ(\x_1)\times_{\bigGroup_\QQ}\gr'_\QQ(\x_2)$ and
    $\gr'_\QQ(\y_1)\times_{\bigGroup_\QQ}\gr'_\QQ(\y_2)$ lie in the
    same $\QQ$-orbit of
      $\AQBigGrSet(\HD_1,\spinc_1)\times_{\bigGroup_\QQ} \DQBigGrSet(\HD_2,\spinc_2)$
    % $\lsub{\HD_1}S_\QQ\times_{\bigGroup_\QQ}\lsub{\QQ}S_{\HD_2}$, 
    and
  \item
    $\gr'_\QQ(\y_1)\times_{\bigGroup_\QQ}\gr'_\QQ(\y_2)=\lambda^{\gr_\QQ(\x,\y)}\gr'_\QQ(\x_1)\otimes_{\bigGroup_\QQ}\gr'_\QQ(\x_2)$.
  \end{enumerate}
\end{theorem}
\begin{proof}
  Since the statements are independent of the base generator used to
  define the grading sets for $\CFAa(\HD_1,\spinc_1)$ and $\CFDa(\HD_2,\spinc_2)$, we
  may choose $\x_i$ to be the base generator for $\HD_i$.

  Since $\spinc(\x)$
  and $\spinc(\y)$ are torsion, it follows
  from~\cite{LipshitzLee08:RelGrad}
  (cf.~Section~\ref{sec:back:Q-grad}) that
  there is a rational domain
  $B\in\pi_2^\QQ(\x,\y)$ connecting $\x$ and $\y$. Intersecting $B$
  with $\HD_1$ and $\HD_2$, we obtain rational domains
  $B_i\in\pi_2^\QQ(\x_i,\y_i)$.

  We argue that the rational domain $B_i$ can be used to compute the
  grading of $\y_i$  (which was originally defined using integral
  domains). 
  Since $\spinc(\x_i)=\spinc(\y_i)$,
  $\pi_2^\QQ(\x_i,\y_i)=\pi_2(\x_i,\y_i)\otimes_\ZZ\QQ$. That is, any
  rational domain $B_i$ connecting $\x_i$ and $\y_i$ can be written as
  \[
  B_i=q_{i,1}C_{i,1}+\cdots+q_{i,\ell}C_{i,\ell}
  \] 
  where the $q_{i,j}\in\QQ$ and the $C_{i,j}\in\pi_2(\x_i,\y_i)$. (To see
  this, note that $\pi_2(\x_i,\y_i)$ is an affine copy of $H_2(Y_i,\bdy
  Y_i;\ZZ)$ while $\pi_2^\QQ(\x_i,\y_i)$ is an affine copy of
  $H_2(Y_i,\bdy Y_i;\QQ)$.)
  Consequently, $B_i$ differs from any integral domain in
  $\pi_2(\x_i,\y_i)$ by a rational periodic domain, and hence has the
  same
  image in  
  $\AQBigGrSet(\HD_1,\spinc_1)$
  % $\lsub{\HD_1}S_\QQ\backslash\bigGroup_\QQ$ 
  or
  $\DQBigGrSet(\HD_2,\spinc_2)$.
  % $\bigGroup_\QQ/\lsub{\QQ}S_{\HD_2}$.
  In formulas, as elements of 
  $\AQBigGrSet(\HD_1,\spinc_1)$
  % $\lsub{\HD_1}S_\QQ\backslash\bigGroup_\QQ$
  and
  $\DQBigGrSet(\HD_2,\spinc_2)$
  % $\bigGroup_\QQ/\lsub{\QQ}S_{\HD_2}$
  respectively,
  \begin{align*}
    g'(B_1)&=(-e(B_1)-n_{\x_1}(B_1)-n_{\y_1}(B_1),\bdy^\bdy
    B_1)=\gr'(\y_1)\\
    \shortintertext{and}
    R(g'(B_2))&=(-e(B_2)-n_{\x_2}(B_2)-n_{\y_2}(B_2),r_*(\bdy^\bdy
    B_2))=\gr'(\y_2).
  \end{align*}
  Note also that $\bdy^\bdy B_2=-\bdy^\bdy B_1$.

  Thus, with our choice of base generator,
  $\gr'_\QQ(\x_1)\times_{\bigGroup_\QQ}\gr'_\QQ(\x_2)=0$ while
  \begin{align*}
    \gr'_\QQ(\y_1)\times_{\bigGroup_\QQ}\gr'_\QQ(\y_2)&=(-e(B_1)-n_{\x_1}(B_1)-n_{\y_1}(B_1)-e(B_2)-n_{\x_2}(B_2)-n_{\y_2}(B_2),0)\\
    &=(-e(B)-n_\x(B)-n_\y(B),0)=\lambda^{\gr_\QQ(\x,\y)},
  \end{align*}
  as desired.
\end{proof}

To complete the computation of the relative $\QQ$-grading on $\CFa$,
we observe that it is always possible to find a splitting satisfying the
conditions of Theorem~\ref{thm:compute-Q-gr}.
\begin{lemma}\label{lem:decomp-exists}
  Given any $3$-manifold $Y$ and torsion $\SpinC$-structures $\spinc$
  and $\spinc'$ on $Y$ there is a decomposition $Y=Y_1\cup_F Y_2$
  of $Y$ along a connected surface $F$ so that
  $\spinc|_{Y_i}=\spinc'|_{Y_i}$ for $i=1,2$.
\end{lemma}
\begin{proof}
  Since a handlebody has a unique $\SpinC$-structure, any Heegaard
  decomposition for $Y$ satisfies the conditions.
\end{proof}

\begin{corollary}\label{cor:bord-gives-gr}
  The $\bigGroup$-set grading $\gr'$ defined in~\cite{LOT1} determines the
  relative $\QQ$-grading on $\HFa$.
\end{corollary}
\begin{proof}
  By definition, the grading $\gr'$ determines $\gr'_\QQ$ which in turn, by
  Lemma~\ref{lem:decomp-exists} and Theorem~\ref{thm:compute-Q-gr},
  determines the relative $\QQ$ grading.
\end{proof}

\begin{remark}\label{rmk:Q-for-small-gr-grp}
  It is sometimes convenient to work with the smaller grading group
  $\smallGroup$ from \cite{LOT1}, rather than $\bigGroup$. To obtain a
  $\smallGroup$-set grading on $\CFDa$ and $\CFAa$, one conjugates by
  grading refinement data; see~\cite[Section 3.1.1]{LOT2}. In the
  proof of Theorem~\ref{thm:compute-Q-gr}, since one works with the
  same grading refinement data on the two sides, it cancels out in the
  computation. Thus, Theorem~\ref{thm:compute-Q-gr} holds with respect
  to the small grading group, as well.
\end{remark}

\begin{remark}
  In~\cite{LOT4}, we give an algorithm for computing $\HFa(Y)$ by
  taking a Heegaard decomposition of $Y$ and factoring the gluing map
  into arc-slides. For such a decomposition, the hypotheses of
  Theorem~\ref{thm:compute-Q-gr} are automatically satisfied. Thus,
  keeping track of the $\bigGroup_\QQ$-gradings along the
  way,~\cite{LOT4} automatically computes the relative $\QQ$-grading
  on $\HFa(Y)$.
\end{remark}

\begin{remark}
  Instead of defining a $\bigGroup_\QQ$-grading on $\CFDa$ by
  (roughly) tensoring $\bigGroup$-grading with $\QQ$ as above, we
  could instead use rational domains to induce a
  $\bigGroup$-grading. The resulting relative grading agrees with the
  one above when the one above is defined, but it is defined more often.
  Theorem~\ref{thm:compute-Q-gr} then no longer needs the hypothesis
  that $\spinc(\x)|_{Y_i}=\spinc(\y)|_{Y_i}$. The drawback is that,
  for this definition, $\gr'_\QQ$ is no longer induced from
  $\gr'$, so one would not obtain Corollary~\ref{cor:bord-gives-gr}.
\end{remark}
 
\section{Examples}\label{sec:examples}

We give an application of Theorem~\ref{thm:compute-Q-gr} to computing the
$\QQ$-graded Heegaard Floer homology groups of surgeries on some knots
in $S^3$. Our knots are rather simple (the unknot and the trefoil),
and hence the graded Heegaard Floer homology groups on their surgeries
have been known for some time; but these computations do give a nice illustration of
the theorem.

% For the torus, we can identify the grading group $G$ with certain
% triples $(m;r,s)\in\QQ^3$ with multiplication given by
% \[
% (m;r,s)\cdot (n;t,u)=(m+n+(ru-st); r+t,s+u).
% \]
% $G$ gradings are induced by a projection $G'\to G$ defined by 
% \[
% (m;a,b,c)\mapsto \left(m; \frac{a+b-c}{2}, \frac{-a+b+c}{2}\right);
% \]
% see~\cite[Section 10.1]{LOT1}.

To start, let $Y$ denote the $(-2)$-framed complement of the left-handed trefoil
$T$. By \cite[Theorem 11.7]{LOT1}, $\CFDa(Y)$ is given by
\[
\begin{tikzpicture}
  \node at (0,0) (x3) {$x_3$};
  \node at (0,-1) (y2) {$y_2$};
  \node at (0,-2) (x2) {$x_2$};
  \node at (2,-2) (y1) {$y_1$};
  \node at (4,-2) (x1) {$x_1.$};
  \draw[->] (x3) to node[left]{$\rho_1$} (y2);
  \draw[->] (x2) to node[left]{$\rho_{123}$} (y2);
  \draw[->] (y1) to node[below]{$\rho_2$} (x2);
  \draw[->] (x1) to node[below]{$\rho_3$} (y1);
  \draw[->] (x1) to node[above, right]{$~~\rho_{12}$} (x3);
\end{tikzpicture}
\]
If we take $x_3$ as the base generator then the gradings lie in
$\bigGroup/\langle(-3/2;-1,1,2)\rangle$, and are given by:
\begin{align*}
\gr(x_1) &= (1;0,2,2)/ \langle (-3/2;-1,1,2) \rangle\\
\gr(x_2) &= (1/2;0,1,1)/\langle (-3/2;-1,1,2) \rangle \\
\gr(x_3) &= (0;0,0,0)/\langle (-3/2;-1,1,2) \rangle \\
\gr(y_1) &= (3/2;0,2,1)/\langle (-3/2;-1,1,2) \rangle \\
\gr(y_2) &= (-1/2;-1,0,0)/\langle (-3/2;-1,1,2) \rangle
\end{align*}
(compare~\cite[Section 10.9]{LOT1}).

% If we take $x_3$ as the base generator then the gradings lie in
% $\bigGroup/\langle(-3/2;-1,2)\rangle$, and are given by:
% \begin{align*}
% \gr(x_1) &= (1;0,2)/ \langle (-3/2;-1,2) \rangle\\
% \gr(x_2) &= (1/2;0,1)/\langle (-3/2;-1,2) \rangle \\
% \gr(x_3) &= (0;0,0)/\langle (-3/2;-1,2) \rangle \\
% \gr(y_1) &= (3/2;1/2,3/2)/\langle (-3/2;-1,2) \rangle \\
% \gr(y_2) &= (-1/2;-1/2,1/2)/\langle (-3/2;-1,2) \rangle.
% \end{align*}
% (See~\cite[Section 10.9]{LOT1} for more details.)

Let $\HD_0$ denote the $\infty$-framed solid torus. Then
$\CFAa(\HD_0)$ has one generator $n$ with $m_3(n,\rho_2,\rho_1)=n$. In
particular,
\[
\gr(n)=\gr(n)\gr(\rho_2)\gr(\rho_1)\lambda=\gr(n)
(-1/2;0,1,0)(-1/2;1,0,0)=\gr(n) (-1/2;1,1,0).
\]
So, $\gr(n)$ lies in $\langle (-1/2;1,1,0)\rangle \backslash \bigGroup$.
% \[
% \gr(n)=\gr(n)\gr(\rho_2)\gr(\rho_1)\lambda=\gr(n)
% (-1/2;1/2,1/2)(-1/2;1/2,-1/2)=\gr(n) (-1/2,1,0).
% \]
% So, $\gr(n)$ lies in $\langle (-1/2;1,0)\rangle \backslash \bigGroup$.

Tensoring the two together, we find that $\CFAa(\HD_0)\DT\CFDa(Y)$ is
generated by $n\otimes y_1$ and $n\otimes y_2$, with no
differential. It follows at once that $\HFa(S^3_{-2}(T))\cong
\Field\oplus \Field$, i.e. $S^3_{-2}(T)$ has the same (ungraded)
Heegaard Floer homology as a lens space; this was, of course, known
before~\cite{OS05:surgeries}.

So far, we have found that the ungraded Heegaard Floer homology of
$-2$ surgery on the trefoil and the unknot are the same. They are,
however, distinguished by their relative $\QQ$-gradings, which we can recover
from the bordered invariants, as follows.

The computation above gives
\begin{align*}
  \gr(n\otimes y_1)&=\langle (-1/2;1,1,0)\rangle\backslash (3/2;0,2,1)/\langle (-3/2;-1,1,2)\rangle\\
  \gr(n\otimes y_2)&=\langle (-1/2;1,1,0)\rangle\backslash(-1/2;-1,0,0)/\langle (-3/2;-1,1,2)\rangle.
\end{align*}
Working in $\bigGroup_\QQ$, we can rewrite the first of these equations as:
\begin{align*}
  \gr(n\otimes y_1)&=\langle (-1/2;1,1,0)\rangle\backslash
  (3/4;-3/2,-3/2,0)\cdot(3/2;0,2,1)\\
  &\qquad\qquad\cdot (3/4;1/2,-1/2,-1)/\langle
  (-3/2;-1,1,2)\rangle\\
  &=\langle (-1/2;1,1,0)\rangle\backslash (1;-1,0,0)/\langle
  (-3/2;-1,1,2)\rangle.
\end{align*}
Consequently, the grading difference between $n\otimes y_1$ and
$n\otimes y_2$ is $3/2$. 

By contrast, the invariant of the $-2$-framed unknot has three
generators:
\[
\begin{tikzpicture}
  \node at (0,0) (b2) {$b_1$};
  \node at (0,-1) (a) {$a$};
  \node at (2,-1) (b1) {$b_2.$};
  \draw[->] (a) to node[below]{$\rho_3$} (b1);
  \draw[->] (a) to node[left]{$\rho_2$} (b2);
  \draw[->] (b1) to node[above]{$~~\rho_{23}$} (b2);
\end{tikzpicture}
\]
If we take $a$ as the base generator, the gradings lie in
$\bigGroup/\langle(1/2;-1,1,2)\rangle$, and are given by:
\begin{align*}
  \gr'(a)&=(0;0,0)/\langle(1/2;-1,1,2)\rangle\\
  \gr'(b_1)&=(-1/2;0,1,0)/\langle(1/2;-1,1,2)\rangle\\
  \gr'(b_2)&=(-1/2;0,0,-1)/ \langle(1/2;-1,1,2)\rangle.
\end{align*}
This gives
\begin{align*}
  \gr'(n\otimes b_1)&=\langle (-1/2;1,1,0)\rangle\backslash
  (-1/2;0,1,0)/\langle(1/2;-1,1,2)\rangle\\
  \gr'(n\otimes b_2)&=\langle (-1/2;1,1,0)\rangle\backslash (-1/2;0,0,-1) / \langle(1/2;-1,1,2)\rangle.
\end{align*}
Working in $\bigGroup_\QQ$, we can rewrite the second of these
equations as:
\[
\gr'(n\otimes b_1)=\langle (-1/2;1,1,0)\rangle\backslash (0;0,0,-1)/ \langle(1/2;-1,1,2)\rangle.
\]
Consequently, the grading difference between $n\otimes b_1$ and
$n\otimes b_2$ is $-1/2$. Thus we see that the relative $\QQ$-grading
distinguishes the Heegaard Floer homology of $-2$ surgery on the
trefoil from $-2$ surgery on the unknot.

\bibliographystyle{../hamsalpha} 
\bibliography{heegaardfloer}
\end{document}

%% file: macros.tex
\newread\testin

\graphicspath{{draws/}{mpdraws/}}
\makeatletter
\def\input@path{{}{draws/}}
\makeatother

\makeatletter
\newcommand\mi@kern[1]{%
  \settowidth\@tempdima{$\mi@obj^{#1}$}
  \kern-\@tempdima
  #1
  \settowidth\@tempdima{$\mi@obj$}
  \kern\@tempdima
}

\newtoks\mi@toksp
\newtoks\mi@toksb
\DeclareRobustCommand{\manyindices}[5]{
  \def\mi@obj{#5}
  \mi@toksp\expandafter{\mi@kern{#2}}
  \mi@toksb\expandafter{\mi@kern{#1}}
  \@mathmeasure4\textstyle{#5_{#1}^{#2}}
  \@mathmeasure6\textstyle{#5_{#3}^{#4}}
  \dimen0-\wd6 \advance\dimen0\wd4
  \@mathmeasure8\textstyle{\hphantom{{}_{#1}^{#2}}#5^{\the\mi@toksp#4}_{\the\mi@toksb#3}}
  \hbox to \dimen0{}{\kern-\dimen0\box8}
}
\makeatother 

\usepackage{ifpdf}
\ifpdf
  
\else
  
\fi

%% file: defs.tex
\newcommand{\ZZ}{\mathbb Z}
\newcommand{\QQ}{\mathbb Q}

\newcommand{\FF}{\mathbb F}

\newcommand{\co}{\nobreak\mskip2mu\mathpunct{}\nonscript
  \mkern-\thinmuskip{:}\penalty300\mskip6muplus1mu\relax}

\newcommand{\bdy}{\partial}

\newcommand{\lbracket}{[}
\newcommand{\rbracket}{]}

\newcommand{\spinc}{\mathfrak s}

\DeclareMathOperator{\divis}{div}

\DeclareMathOperator{\spin}{spin}
\newcommand{\SpinC}{\spin^c}

\DeclareMathOperator{\gr}{gr}

\theoremstyle{plain}
\newtheorem{theorem}{Theorem}
\numberwithin{equation}{section}
\newtheorem{citethm}[equation]{Theorem}

\newtheorem{lemma}[equation]{Lemma}
\newtheorem{corollary}[equation]{Corollary}

\newtheorem{definition}[equation]{Definition}

\theoremstyle{definition}

\theoremstyle{remark}

\newtheorem{remark}[equation]{Remark}
\newcommand{\HF}{\mathit{HF}}
\newcommand{\HFa}{\widehat {\HF}}
\newcommand{\HFm}{{\HF}^-}
\newcommand{\HFp}{{\HF}^+}
\newcommand{\HFinf}{{\HF}^\infty}

\newcommand{\CFa}{\widehat {\mathit{CF}}}

\newcommand{\x}{\mathbf x}
\newcommand{\y}{\mathbf y}
\newcommand{\z}{\mathbf z}

\newcommand\HH{\mathit{HH}}

\newcommand\Hochschild\HH

\newcommand{\Ainf}{\mathcal A_\infty}

\newcommand{\Alg}{\mathcal{A}}

\newcommand{\alphas}{{\boldsymbol{\alpha}}}
\newcommand{\betas}{{\boldsymbol{\beta}}}

\newcommand{\CFD}{\mathit{CFD}}

\newcommand{\CFA}{\mathit{CFA}}

\newcommand{\CFDa}{\widehat{\CFD}}

\newcommand{\CFAa}{\widehat{\CFA}}

\newcommand{\cZ}{\mathcal{Z}}
\newcommand{\PtdMatchCirc}{\cZ}
\newcommand{\PMC}{\PtdMatchCirc}

\newcommand{\dg}{\textit{dg} }

\newcommand\DTP{\mathbin{\widetilde\otimes}}
\newcommand\DT{\boxtimes}
\newcommand\Gen{\mathfrak{S}}

\newcommand{\Field}{{\FF_2}}

\newcommand{\Heegaard}{\mathcal{H}}
\newcommand{\HD}{\Heegaard}

\newcommand{\bigGroup}{G'}
\newcommand{\smallGroup}{G}
\newcommand{\ABigGrSet}{G'_A}
\newcommand{\AQBigGrSet}{G'_{A,\QQ}}
\newcommand{\DBigGrSet}{G'_D}
\newcommand{\DQBigGrSet}{G'_{D,\QQ}}

\makeatletter
\newcommand\honestalg[3]{\bigl\lbracket
\begin{smallmatrix} #1\@ifempty{#3}{}{&#3} \\ #2 \end{smallmatrix}
\bigr\rbracket}
\makeatother
\newcommand{\lsub}[2]{{}_{#1}#2}
\newcommand{\lsup}[2]{{}^{#1}\mskip-.6\thinmuskip#2}